\documentclass[12pt, a4paper]{article}

    \usepackage{amsmath, amsthm, amssymb, amsfonts}
    \usepackage[rm,bf,compact,topmarks,calcwidth,pagestyles]{titlesec}
    \usepackage{titletoc}
    \usepackage{multicol}
    \usepackage{bbm}
    \usepackage{setspace}
    \usepackage{graphicx}
    \usepackage[bf,small,center]{caption}

    \newtheorem{thm}{Theorem}[section]
    \newtheorem{cor}[thm]{Corollary}
    \newtheorem{lem}[thm]{Lemma}

    \theoremstyle{remark}
    \newtheorem{rem}[thm]{Remark}

    \newcommand{\ga}{\alpha}
    \newcommand{\gb}{\beta}
    \newcommand{\gd}{\delta}
    \newcommand{\ep}{\varepsilon} % can't use \ge
    \newcommand{\gm}{\gamma} % can't use \gg
    \newcommand{\gk}{\kappa}
    \newcommand{\gl}{\lambda}
    \newcommand{\go}{\omega}
    
    \newcommand{\gs}{\sigma}
    \newcommand{\gt}{\theta}
    
    \newcommand{\gz}{\zeta}

    \newcommand{\Exp}{\mathbb{E}}

    \newcommand{\ito}{It\^o}

    \newcommand{\ind}{\mathbbm{1}}      %{1\hspace{-2.5mm}{1}}

    \renewcommand{\over}[1]{\overline{#1}_{t_0}}
    \renewcommand{\Pr}{\mathbb{P}}

    \newcommand{\Qr}{\mathbb{Q}}

    \newcommand{\km}{K^{-}}
    \newcommand{\kp}{K^{+}}
    \newcommand{\ks}{K^{*}}
    \newcommand{\ka}{\hat{K}}

    \newcommand{\ov}[1]{\overline{#1}}
    \newcommand{\un}[1]{\underline{#1}}

    \newcommand{\xb}{\overline{x}}

    \newcommand{\ms}{\check{\mu}}
    \newcommand{\mh}{\hat{\mu}}
    \newcommand{\mlow}{\underline{\mu}}
    \newcommand{\mup}{\overline{\mu}}

    \newcommand{\Dep}{D_{\ep}}
    \newcommand{\Depm}{D_{\ep}^-}

    \newcommand{\tp}{t_0}

    \newcommand{\Qh}{{\Qr}^z_x}
    \newcommand{\Eh}{\hat{\Exp}_x}
    \newcommand{\Eqh}{\hat{\Exp}^z_x}
    \newcommand{\Ph}{{\Pr}^z_x}
    \newcommand{\Prb}[2]{\emph{P} \left( #1, #2 \right)}
    \newcommand{\Prbb}[2]{\emph{P} \left( #1, #2; B \right)}
    \newcommand{\Prbm}[2]{P \left( #1, #2 \right)}  % Use in theorem type environments
    \newcommand{\Prbbm}[2]{P \left( #1, #2; B \right)} % Use in theorem type environments

    \newcommand{\LM}{M^*(t)}
    \newcommand{\LB}{B^*(g,t)}
    \newcommand{\LG}{g^*_t(s)}
    \newcommand{\LL}{L^*(t)}
    \newcommand{\LC}{C^*(g,t)}

    \newcommand{\bmv}{\mbox{\boldmath$W$}}%{\mbox{\emph{$\mathbf{W}$}}}
    \newcommand{\mbf}[1]{\mbox{\boldmath$#1$}}
    \newcommand{\mbff}[1]{\mbox{\footnotesize\boldmath$#1$}}
    \newcommand{\norm}[1]{\left|\left|#1\right|\right|}

    \newcommand{\gp}{g_{+}}
    \newcommand{\gi}{g_{-}}
    \newcommand{\fp}{f_{+}}
    \newcommand{\fm}{f_{-}}

    \renewcommand{\cal}[1]{{\mathcal #1}}

\begin{document}

    \title{First Passage Densities and Boundary Crossing Probabilities for Diffusion Processes}
    \author{A.N. Downes\footnote{Department of Mathematics and Statistics, University of Melbourne, a.downes@ms.unimelb.edu.au}~ and K. Borovkov\footnote{Department of Mathematics and Statistics, University of Melbourne}}
    \date{}

    \maketitle

\begin{abstract}
    We consider the boundary crossing problem for time-homogeneous diffusions and general curvilinear boundaries. Bounds are derived for the approximation error of the one-sided (upper) boundary crossing probability when replacing the original boundary by a different one. In doing so we establish the existence of the first-passage time density and provide an upper bound for this function. In the case of processes with diffusion interval equal to $\mathbb{R}$ this is extended to a lower bound, as well as bounds for the first crossing time of a lower boundary. An extension to some time-inhomogeneous diffusions is given. These results are illustrated by numerical examples.
\end{abstract}


\section{Introduction}

    Calculating the probability a given diffusion process will cross a given one- or two-sided boundary in a finite time interval is an important problem for which no closed-form solution is known except for a few special cases. Such boundary crossing probabilities arise in many applications, including finance (pricing of barrier options) and sequential statistical analysis. Since no closed form exact solution is immediately forthcoming in the general case, finding approximate solutions, together with approximation rates, becomes an important alternative approach. A possible pathway here is to replace the original boundary with a close one, for which obtaining boundary crossing probabilities is feasible, and then to obtain a bound for the error caused by using the approximating boundary.

    This approach was taken in \cite{Borovkov_etal_0305}, where an error bound was established for boundary crossing probabilities for the Brownian motion. It was shown there that under mild regularity conditions, the difference between the probabilities doesn't exceed a multiple of the uniform distance between the original and approximating boundaries, with the coefficient being an explicit function of the Lipschitz coefficient of the boundary. An earlier (asymptotic) bound, again in the Brownian motion case, was obtained in \cite{Potzelberger_etal_0301}, but under some additional superfluous conditions. In \cite{Wang_etal_0307}, the authors dealt with a special class of diffusion processes which can be expressed as piecewise monotone (not necessarily one-to-one) transformations of the standard Brownian motion. For such processes, boundary crossing problems could be reduced to similar ones for the Brownian motion process. Further references can be found in \cite{Potzelberger_etal_0301} and \cite{Wang_etal_0307}.

    Much work has also been done in the area of calculating the density of the first crossing time. Explicit formulae can be obtained for a limited number of specific pairs of diffusions and boundaries, such as in \cite{Giorno_etal_0389}, \cite{Daniels_0696}, \cite{Gutierrez_etal_0997} (see also references therein), where the existence of the density is usually assumed. Existence of the first crossing density was established in \cite{Fortet_xx43} under certain conditions on the diffusion and boundary. For flat boundaries and under certain smoothness conditions on the diffusion and drift coefficients (which, for example, do not apply to the Bessel process), \cite{Pauwels_0687} establishes the existence and smoothness of the first-passage density.

    In this paper we extend the work of \cite{Borovkov_etal_0305} to general diffusion processes. As an auxiliary result (which is also of independent interest), we establish the existence of, and bounds for, the density of the first passage time under mild conditions on the boundary and diffusion. Most of this work is focused on the time-homogeneous case, however we outline how some time-inhomogeneous processes can be treated as well. Section~\ref{sec:setup} introduces the processes involved and some other notation. Section~\ref{sec:fpt_density} proves the existence of a density for the first passage time, and provides upper and lower bounds for the density. This is used in Section~\ref{sec:bcp} to bound the difference between the crossing probabilities of two close boundaries. Section~\ref{sec:examples} provides some numerical examples.

\section{Setup and Notation}
\label{sec:setup}

    First we consider time-homogeneous diffusions; a possible generalisation is discussed at the end of Section~\ref{sec:fpt_density}. We begin with a diffusion $U_t$ governed by the stochastic differential equation (SDE)
        \begin{align}
        \label{eq:original_diff}
            dU_t = \nu (U_t) dt + \gs(U_t) dW_t,
        \end{align}
    where $W_t$ is a Brownian motion and $\gs(y)$ is differentiable and non-zero inside the diffusion interval (that is, the the smallest interval $I \subseteq \mathbb{R}$ such that $X_t \in I$ almost surely). As is well-known, one can transform the process to one with unit diffusion coefficient. This is achieved by (see e.g. \cite{Rogers_xx85}, p.161) defining
        \begin{align*}
        %\label{eq:fn_transform}
            F(y) := \int_{y_0}^y \frac{1}{\gs(u)} du
        \end{align*}
    for some $y_0$ from the diffusion interval of $U_t$, and then considering $X_t := F(U_t)$. By \ito's formula, this process will have unit diffusion coefficient and a drift coefficient $\mu(y)$ given by the composition
        \[  \mu(y) = \left( \frac{\nu}{\gs} - \frac{1}{2} \gs'\right) \circ F^{-1}(y).  \]
    Any boundaries being considered must also be transformed accordingly. From here on we work with the transformed diffusion process $X_t$ governed by the SDE
        \begin{align}
        \label{eq:Xt_diff}
            dX_t = \mu(X_t) dt + dW_t, \;\; X_0 = x.
        \end{align}
    Conditions mentioned throughout refer to the transformed process $X_t$ and its drift coefficient $\mu$.

    Without loss of generality, we consider boundary crossing probabilities over a time interval $[0,1]$. The general case of an interval $[0,T]$ follows by a change of scale.

    An `upper' boundary for the process $X_t$ is a function $g(t)$ such that $g(0)>x$, while a `lower' boundary satisfies $g(0)<x$. When considering a two-sided boundary crossing problem with an upper boundary $\gp$ and lower boundary $\gi$, we assume that $\gp(t) > \gi(t)$ for $t \in [0,1]$.

    We will consider the following two cases only:
        \begin{enumerate}
            \item[] [A] The diffusion interval of $X_t$ is the whole real line $\mathbb{R}$.
            \item[] [B] The diffusion interval of $X_t$ is $\mathbb{R}_+ = [0, \infty)$.
        \end{enumerate}
    Where the specific interval is not important, the lower end-point (either $-\infty$ or $0$) will be denoted by $\ell$. When considering boundary-crossing problems for diffusions on $\mathbb{R}_+$, we assume the boundaries do not touch zero in the time interval $[0, 1]$. The results extend to diffusions with other diffusion intervals with one finite endpoint by employing appropriate transforms.

    We denote by $\Pr_x$ probabilities conditional on the process in question ($X_t$ or some other process, which will be obvious from the context) starting at $x$. Where no subscript is present, either conditioning is mentioned explicitly or the process is assumed to start from zero (the latter will only apply to the Brownian motion process).

    For the diffusion $X_t$ we will need a `reference' diffusion with certain characteristics. These are: being able to bound the probability that this reference diffusion crosses a linear boundary when its final value is known; the process has unit diffusion coefficient; and the process has the same diffusion interval as $X_t$. If $X_t$ has a diffusion interval of $\mathbb{R}$, we use the Brownian motion as the reference process. For diffusions on $\mathbb{R}_+$ we use the Bessel process of an arbitrary integer-valued dimension $d \geq 3$. Note that a further reason for using Bessel processes is that they are of independent interest, having applications in many areas.

    Recall the definition of the Bessel process $R_t$ of dimension $d = 1, 2, \ldots$. This process gives the Euclidean distance from the origin of the $d$-dimensional Brownian motion, that is,
        \[  R_t = \sqrt{\bigl(W_t^{(1)}\bigr)^2 + \cdots + \bigl(W_t^{(d)}\bigr)^2},    \]
    where the $W_t^{(i)}$ are independent standard Brownian motions, $i = 1, \ldots, d$. As is well known (see e.g. \cite{Revuz_etal_xx99}, p.445), $R_t$ satisfies the SDE
        \begin{align}
        \label{eq:Rt_diff}
            dR_t = \frac{d-1}{2} \frac{1}{R_t}dt + dW_t.
        \end{align}
    Note that for non-integer values of $d$ the Bessel process of order $d$ is defined using the above SDE. The process has a transition density function $p_R(t, y, z)$ given by
        \begin{align}
            \label{eq:bes_tpdf}
        p_R(t,y,z) = z \left(\frac{z}{y}\right)^{\eta} t^{-1} e^{-(y^2 + z^2)/2t} {\cal I}_{\eta} \left(\frac{yz}{t} \right),
        \end{align}
    where $\eta = d/2 -1$ and ${\cal I}_{\eta}(z)$ is the modified Bessel function of the first kind. For further information, see \cite{Revuz_etal_xx99}.

    Finally, by $\Phi$ we denote the standard normal distribution function.

\section{The Existence and Bounds of the First Passage Time Densities}
\label{sec:fpt_density}

    This section proves the existence of, and gives an upper bound for, the density of the first passage time
        \[\tau := \inf\{t>0 : X_t \geq g(t)\}. \]
    The diffusion $X_t$, with $X_0 = x$, is subject to some mild conditions and the upper boundary $g(t)$ satisfies Lipschitz condition. For diffusions on $\mathbb{R}$ the result is then extended to lower boundaries, and further results give lower bounds for the first hitting time density for both upper and lower boundaries.

    Let $\ov{g}(t)$ denote the maximum of the boundary function up to time $t$:
        \[  \ov{g}(t) := \max_{0 \leq s \leq t} g(s).   \]
    We assume $X_t$ has a transition density $p(t,y,z)$ and define functions $q(t,y,z)$, $G(y)$ and $M(t)$ as follows, according to the diffusion interval of $X_t$:
        \begin{enumerate}
            \item[][A] If the diffusion interval of $X_t$ is $\mathbb{R}$, then we denote by $q(t, y, z)$ the transition density of the Brownian motion, and define, for some fixed $y_0 \in \mathbb{R}$,
                    \begin{equation}
                        \begin{array}{rl}
                            G(y) &:= \displaystyle\int_{y_0}^{y} \mu(z) dz,\\
                            M(t) &:= \displaystyle\inf_{-\infty < y \leq \ov{g}(t)} \left(\mu'(y) + \mu^2(y)\right).
                        \end{array}
                    \label{eq:M_R}
                    \end{equation}
            \item[][B] If the diffusion interval of $X_t$ is $\mathbb{R}_+$, then we denote by $q(t, y, z)$ the transition density of a $d$-dimensional Bessel process for some fixed integer $d \geq 3$, and define, for some fixed $y_0\geq 0$,
                    \begin{equation}
                        \begin{array}{rl}
                            G(y) &:= \displaystyle\int_{y_0}^{y} \left(\mu(z) - \frac{d-1}{2z} \right) dz,\\
                            M(t) &:= \displaystyle\inf_{0 < y \leq \ov{g}(t)} \left( \mu'(y) - \frac{(d-1)(d-3)}{4y^2} + \mu^2(y)  \right).
                        \end{array}
                    \label{eq:M_Rplus}
                    \end{equation}
        \end{enumerate}
    To simplify the statement of our main theorem, we set
        \begin{equation}
        \label{eq:x_bar}
        \xb := \begin{cases} x & \mbox{in case [A]} \\
                                                        -x & \mbox{in case [B]}.
                         \end{cases}
        \end{equation}

    \begin{thm}
    \label{th:density_bound}
        Let $X_t$ be a diffusion process satisfying \eqref{eq:Xt_diff} with $\mu(y)$ differentiable inside the diffusion interval.  Let $g(t)$ be a function on $[0,1]$ such that $g(0)>x$ and, for some $K^{\pm} \in \mathbb{R}$,
            \begin{align}
            \label{eq:lip_assump}
                -\km h \leq g(t+h) - g(t) \leq \kp h, \;\;\; 0 \leq t < t+h \leq 1.
            \end{align}
        Then the first passage time $\tau$ of the process $X_t$ of the boundary $g(t)$, $t \in(0,1)$, has a density $p_{\tau}(t)$, satisfying
            \begin{align}
            \label{eq:density_bound}
                p_{\tau}(t) \leq B(g,t),
            \end{align}
        where $B(g,t)$ is given by
            \begin{align}
            \label{eq:Bt_def}
                B(g,t) := \frac{1}{t} (g(t) + \km t - \xb) q(t,x,g(t)) e^{G(g(t)) - G(x) -(t/2) M(t)}.
            \end{align}
    \end{thm}

    \begin{rem}
        \begin{enumerate}
            \item[i)] For diffusions on $\mathbb{R}_+$, the choice of the integer $d$ is arbitrary, subject to $d  \geq 3$, and can be used to optimise the bound.
            \item[ii)] The bound is sharp: for constant $\mu$ and a linear boundary, equality holds in \eqref{eq:density_bound}.
            \item[iii)] In terms of the original process $U_t$ in \eqref{eq:original_diff}, the condition of differentiability of $\mu(y)$ requires $\nu(z)$ to be differentiable and $\gs(z)$ to be twice differentiable.
        \end{enumerate}
    \end{rem}

    \begin{proof}
    For a fixed $t$ and small $h>0$, we can write, conditioning on the value of $X_t$,
        \begin{align}
        \label{eq:split_tau}
            \Pr_x (\tau \in (t, t+h)) &= \int_{\ell}^{g(t)} \Pr_x (\tau \in (t, t+h) | X_t = z) \Pr_x (X_t \in dz)\notag\\
                    &= \int_{\ell}^{g(t)-h^{1/4}} + \int_{g(t)-h^{1/4}}^{g(t)} =: I_1 + I_2.
        \end{align}
    Clearly, by the Markov property
        \begin{align}
        \label{eq:break_tau_cross}
            \Pr_x (\tau \in (t, t+h) | X_t = z) = & \Pr_x \left( \sup_{0\leq s \leq t} \left(X_s - g(s)\right) < 0 \Big| X_t = z \right)\notag\\
            & \times \Pr_x \left( \sup_{t\leq s \leq t+h} \left(X_s - g(s)\right) \geq 0 \Big| X_t = z \right).
        \end{align}
    We proceed by bounding the first factor on the right hand side for the region $g(t) - h^{1/4} \leq z \leq g(t)$ (it is not required for the remaining interval, as will be seen later). The process in question is now a `diffusion bridge', or `pinned diffusion': we are interested in the behavior of the original process $X_s$ over $0 \leq s \leq t$ conditional on its value at time $t$. We will use the following lemma.

    \begin{lem}
    \label{lem:bridge_prob}
    Let $\Ph$ denote the law of $X_s$ starting at $X_0 = x$ and pinned by $X_t = z$. If $\cal{F}_s := \gs\left(X_u: u \leq s\right)$, then for any $A \in \cal{F}_t$
            \[ \Ph(A) = \frac{q(t, x, z)}{p(t, x, z)} e^{G(z) - G(x)}   \Eqh\left[e^{-(1/2) N(t)}\ind_{\{X \in A\}}\right], \]
        where $\ind_B$ is the indicator of the event $B$ and $\Eqh$ denotes expectation with respect to the probability $\Qh$ defined as follows:
        \begin{enumerate}
            \item[{\em (i)}] if the diffusion interval of $X_s$ is $\mathbb{R}$, then $\Qh$ denotes the law of the Brownian motion $W_s$ on $[0, t]$, starting at $W_0 = x$ and conditional on $W_t = z$, while
                \[  N(t) := \int_0^t \left(\mu'(X_u) + \mu^2(X_u)\right) du;    \]
            \item[{\em (ii)}] if the diffusion interval of $X_s$ is $[0, \infty)$, then $\Qh$ denotes the law of the Bessel processes $R_s$ of dimension $d \geq 3$ on $[0, t]$, starting at $R_0 = x$ and conditional on $R_t = z$, while
                \[  N(t) := \int_0^t \left( \mu'(X_u) - \frac{(d-1)(d-3)}{4X_u^2} + \mu^2(X_u)  \right) du. \]
        \end{enumerate}
    \end{lem}

    \begin{proof}
        In case [A], the result was obtained in \cite{Baldi_etal_0802}, so we only need to consider case [B]. For the latter, we will use a modification of the argument leading to (14) in \cite{Baldi_etal_0802}.

        Let $\cal{C} = \cal{C}([0, t], \mathbb{R}^+)$ be the space of all continuous positive functions on $[0, t]$, which we will interpret as a sample space for the processes in question, and on $\cal{C}$ define $V_s(\go) = \go_s$, $\go \in {\cal C}$, ${\cal{F}}_s = \gs(V_u, u \leq s)$. On the space $(\cal{C}, \cal{F}_t)$, in addition to the laws $\Pr_x$ and $\Ph$, define
            \begin{align*}
                \Qr_x & := \mbox{ the law of }R_s,\\
                \Qh & := \mbox{ the law of }R_s \mbox{, pinned by }R_t = z,
            \end{align*}
        where $R_s$ is the $d$-dimensional Bessel process starting at $R_0 = x$ (see \eqref{eq:Rt_diff}). Hence under $\Pr_x$
            \[  dV_s = \mu(V_s) ds + dW_s,  \]
        while under $\Qr_x$
            \[  dV_s = \frac{d-1}{2V_s} ds + d\widetilde{W}_s,  \]
        for a $\mathbb{Q}$ Brownian motion $\widetilde{W}_s$. Let $\Eh$ denote expectation with respect to $\Qr_x$. If
            \begin{align*}
                \gz_s &:= \exp\left\{ \int_0^s \left( \mu(V_u) - \frac{d-1}{2V_u} \right) dW_u - \frac{1}{2} \int_0^s \left( \mu(V_u) - \frac{d-1}{2V_u} \right)^2 du \right\}\\
                    &=\exp\Bigg\{ \int_0^s \left( \mu(V_u) - \frac{d-1}{2V_u} \right)\left(dV_u - \frac{d-1}{2V_u}du \right) \\
                    &\hspace{8cm} - \frac{1}{2} \int_0^s \left( \mu(V_u) - \frac{d-1}{2V_u} \right)^2 du \Bigg\},
            \end{align*}
        then by Girsanov's theorem, for all $A \in \cal{F}_t$ we have
            \[ \Pr_x(A) = \Eh \left[ \gz_t \ind_{\{V \in A\}} \right].  \]
        Under $\Qr_x$ we have by \ito's formula, using \eqref{eq:M_Rplus}, that
            \[  \int_{0}^{t} \left(\mu(V_u) - \frac{d-1}{2V_u} \right)dV_u = G(V_t) - G(V_0) - \frac{1}{2} \int_{0}^{t} \left(\mu'(V_u) + \frac{d-1}{2V_u^2} \right) du,    \]
        so
            \begin{align}
            \label{eq:gir_density}
                \gz_t &= \exp \left\{ G(V_t) - G(V_0) - \frac{1}{2} \int_{0}^{t} \left( \mu'(V_u) + \frac{d-1}{2V_u^2} + \mu^2(V_u) - \left(\frac{d-1}{2V_u}  \right)^2 \right) du \right\}\notag\\
                          &= \exp \left\{ G(V_t) - G(V_0) - \frac{1}{2} \int_{0}^{t} \left( \mu'(V_u) - \frac{(d-1)(d-3)}{4V_u^2} + \mu^2(V_u)  \right) du \right\}.
            \end{align}
        Using Lemma~3.1 of \cite{Baldi_etal_0802}, we have
            \begin{align}
            \label{eq:prob_rel}
                \Ph(A) = \frac{q(t,x,z)}{p(t,x,z)} \Eqh [ \gz_t \ind_{\{V \in A\}} ].
            \end{align}
        Inserting \eqref{eq:gir_density} into \eqref{eq:prob_rel} gives
            \begin{align*}
                \Ph(A) &= \frac{q(t, x, z)}{p(t, x, z)} \Eqh \Bigg[ \exp \Bigg\{G(V_t) - G(V_0) \\
                 & \hspace{2.5cm}- \frac{1}{2} \int_{0}^{t} \left( \mu'(V_u) - \frac{(d-1)(d-3)}{4V_u^2}  + \mu^2(V_u)  \right) du \Bigg\} \ind_{\{V \in A\}} \Bigg],
            \end{align*}
        completing the proof of the lemma.
    \end{proof}

    Next we apply Lemma \ref{lem:bridge_prob} to our boundary crossing problem.

    \begin{lem}
    \label{lem:pinned_diffs}
        \begin{enumerate}
            \item[{\em (i)}] In case [A],
                \begin{align*}
                    \Pr_x \bigg( \sup_{0\leq s \leq t} (X_s &- g(s)) < 0 \Big| X_t = z \bigg)\\
                         &\leq \frac{q(t, x, z)}{p(t, x, z)} e^{G(z) - G(x) -(t/2) M(t)}    \Pr_x \left(\sup_{0 \leq s \leq t} \left(W_s - g(s)\right) < 0 \Big| W_t = z \right);
                \end{align*}
            \item[{\em (ii)}] in case [B],
                \begin{align*}
                    \Pr_x \bigg( \sup_{0\leq s \leq t} (X_s &- g(s)) < 0 \Big| X_t = z \bigg)\\
                        &\leq \frac{q(t, x, z)}{p(t, x, z)} e^{G(z) - G(x) -(t/2) M(t)} \Pr_x \left(\sup_{0 \leq s \leq t} \left(R_s - g(s)\right) < 0 \Big| R_t = z \right),
                \end{align*}
        \end{enumerate}
        where $M(t)$ was defined in \eqref{eq:M_R}, \eqref{eq:M_Rplus} according to the diffusion interval of $X_s$.
    \end{lem}
    The assertions of the lemma immediately follow from that of Lemma~\ref{lem:bridge_prob} with $A = \left\{ \sup_{0\leq s \leq t} \left(X_s - g(s)\right) < 0 \right\}$. Note that $M(t)$ gives a lower bound on the integrand in $N(t)$ for all paths in $A$.

    In order to apply the above lemma, we introduce the linear boundary
        \begin{align}
        \label{eq:lin_bdy}
            g_t(s) := g(t) + \km(t-s),  \;\;\; 0 \leq s \leq 1.
        \end{align}
    Due to assumption~\eqref{eq:lip_assump}, we have $g_t(s) \geq g(s)$ for $0 \leq s \leq t$, and hence
        \begin{align*}
            \Pr_x \left(\sup_{0 \leq s \leq t} \left(Y_s - g(s)\right) < 0 \Big| Y_t = z \right) & \leq \Pr_x \left(\sup_{0 \leq s \leq t} \left(Y_s - g_t(s)\right) < 0 \Big| Y_t = z \right)
        \end{align*}
    for any process $Y_s$. We now need to bound the boundary crossing probabilities for Brownian bridges and pinned Bessel processes for linear boundaries. In the case of Brownian bridges, an explicit formula is well-known (see e.g. \cite{Borodin_0505}, pp. 64-67):
        \begin{align}
        \label{eq:BBridge}
            \Pr_x \left(\sup_{0 \leq s \leq t} \left(W_s - g_t(s)\right) < 0 \Big| W_t = z \right) = 1 - e^{-2 b_1(x) b_2(z)},
        \end{align}
    where $b_1(x) = (g_t(0) - x)$ and $b_2(z) = t^{-1}\left(g_t(0) - K^-t - z \right)$. The pinned Bessel process satisfies the following inequality:

    \begin{lem}
    \label{lem:BesBridge}
        For a $d$-dimensional Bessel process $R_s$, $d=1,2,\ldots$, and $x \in (0, g(0))$, $z \in (0, g(t))$,
            \begin{align*}
                \Pr_x \left(\sup_{0 \leq s \leq t} \left(R_s - g_t(s)\right) < 0 \Big| R_t = z \right) \leq 1 - e^{-2 b_1(-x) b_2(z)},
            \end{align*}
        with $b_1(x)$ and $b_2(z)$ defined above.
    \end{lem}

    \begin{rem}
        If the conditioned diffusion (or `Bessel bridge') $r_s$, $0 \leq s \leq t$, in Lemma~\ref{lem:BesBridge} (whose distribution coincides with the conditional law of $R_s$, $0 \leq s \leq t$, given $R_t = z$) followed, by analogy with the case of the Brownian bridge, a SDE of the form
            \[  dr_s = \left(\frac{z - r_s}{t-s} + \frac{d-1}{2r_s} \right) ds + dW_s   \]
        (as stated in \cite{Alili_etal_xx05}), we could obtain a sharper bound than the one from Lemma~\ref{lem:BesBridge}. Unfortunately, it appears that the Bessel bridge process actually satisfies the following SDE (which can easily be derived from Theorem~5.8 and (4.d1)--(4.d3) in \cite{Pitman_etal_x180}):
            \begin{align}
            \label{eq:BeB_sde}
                dr_s = \frac{1}{t-s} \left( z \frac{{\cal I}_{\eta+1} \left(\frac{z r_s}{t-s}\right)}{{\cal I}_{\eta} \left(\frac{z r_s}{t-s}\right)} - r_s \right) ds + \frac{d-1}{2r_s} ds + dW_s
            \end{align}
        (cf. \eqref{eq:bes_tpdf}). Alternatively, this SDE could also be derived using an explicit formula for the transition density of the Bessel bridge:
            \[  \Pr_x \left(R_s \in dy \big| R_u = v, R_t = z \right) =  y \frac{t-u}{(s-u)(t-s)} \frac{e^{-\frac{v^2 + y^2}{2(s-u)}} e^{-\frac{z^2 + y^2}{2(t-s)}}}{ e^{-\frac{v^2 + z^2}{2(t-u)}}}\frac{{\cal I}_{\eta} \left(\frac{v y}{s-u}\right) {\cal I}_{\eta} \left(\frac{z y}{t-s}\right)}{{\cal I}_{\eta} \left(\frac{v z}{t-u}\right)} dy,  \]
    $u < s < t$, which is easily obtainable from (1.0.6) of \cite{Borodin_etal_xx02}, p.373. With \eqref{eq:BeB_sde} we are unable to improve the above result.
    \end{rem}

    \begin{proof}
    Let $\bmv_s$ denote the $d$-dimensional Brownian motion $\bigl(W_s^{(1)}, W_s^{(2)}, \ldots, W_s^{(d)}\bigr)$ with initial value $\bmv_0 = \mbf{x} = (x, 0, \ldots, 0)$, so we can assume without loss of generality that $R_s = \norm{\bmv_s}$, $s \ge 0$. Then
        \begin{align*}
            \Pr_x \bigg(\sup_{0 \leq s \leq t} &\left(R_s - g_t(s)\right) < 0 \Big| R_t = z \bigg) \\
            &= \Pr_{\mbff{x}} \left(\sup_{0 \leq s \leq t} \left(\norm{\bmv_s} - g_t(s)\right) < 0 \Big| \norm{\bmv_t} = z \right)\\
            &= \int_{\mbff{z}:\, \norm{\mbff{z}} = z} \Pr_{\mbff{x}} \left(\sup_{0 \leq s \leq t} \left(\norm{\bmv_s} - g_t(s)\right) < 0 \Big| \bmv_t = \mbf{z} \right)\\
            & \qquad \times \Pr_{\mbff{x}} \left( \bmv_t \in d\mbf{z} \Big| \norm{\bmv_t}=z \right)\\
            &\leq \sup_{\mbff{z},\mbff{y}:\, \norm{\mbff{z}} = z, \norm{\mbff{y}}=x} \Pr_{\mbff{y}} \left(\sup_{0 \leq s \leq t} \left(\norm{\bmv_s} - g_t(s)\right) < 0 \Big| \bmv_t = \mbf{z} \right).
        \end{align*}
    Due to the rotational symmetry of the $d$-dimensional Brownian motion, we may rotate our coordinates so that $\mbf{z} = \left(z, 0, \ldots, 0\right)$. This gives
        \begin{align*}
            \Pr_x \bigg(\sup_{0 \leq s \leq t} &\left(R_s - g_t(s)\right) < 0 \Big| R_t = z \bigg) \\
            &\leq \sup_{\mbff{y}:\, \norm{\mbff{y}}=x} \Pr_{\mbff{y}} \left(\sup_{0 \leq s \leq t} \left(\norm{\bmv_s} - g_t(s)\right) < 0 \Big| \bmv_t = (z, 0, \ldots, 0) \right)\\
            &\leq \sup_{y:\, |y| \leq x} \Pr_{y} \left(\sup_{0 \leq s \leq t} \left(W_s^{(1)} - g_t(s)\right) < 0 \Big| W_t^{(1)} = z \right)\\
            &= \Pr_{-x} \left(\sup_{0 \leq s \leq t} \left(W_s^{(1)} - g_t(s)\right) < 0 \Big| W_t^{(1)} = z \right).
        \end{align*}
    Applying \eqref{eq:BBridge} concludes the proof.
    \end{proof}

    With $\xb$, $M(t)$ and $G$ defined in \eqref{eq:M_R}--\eqref{eq:x_bar} and for $g(t) - h^{1/4} \leq z \leq g(t)$ we thus have
        \begin{align}
        \label{eq:final_before}
            \Pr_x \bigg( \sup_{0\leq s \leq t} \left(X_s - g(s)\right) &< 0 \Big| X_t = z \bigg) \leq \frac{q(t, x, z)}{p(t, x, z)} e^{G(z) - G(x) -(t/2) M(t)} (1 - e^{-2 b_1(\xb) b_2(z)})\notag\\
            &\leq  \sup_{g(t) - h^{1/4} \leq w \leq g(t)} q(t,x,w) e^{G(w) - G(x) -(t/2) M(t)} \frac{2b_1(\xb) b_2(z)}{p(t,x,z)}\notag\\
            &= 2 A_h(g,t) \frac{g(t) - z}{p(t,x,z)},
        \end{align}
    where
        \[  A_h(g,t) := \sup_{g(t) - h^{1/4} \leq w \leq g(t)} q(t,x,w) e^{G(w) - G(x) -(t/2) M(t)} \frac{1}{t} (g_t(0) - \xb). \]
    Now return to \eqref{eq:break_tau_cross} and consider the second factor, i.e. the probability that the process $X_s$ crosses the boundary $g(s)$ during $(t, t+h)$ given $X_t = z$. Again we bound this probability by the corresponding linear boundary crossing probability, with the linear boundary defined in \eqref{eq:lin_bdy}. Due to assumption \eqref{eq:lip_assump}, we have $g_t(s) \leq g(s)$ for $t \leq s \leq 1$, and hence
        \[  \Pr_x \left( \sup_{t\leq s \leq t+h} \left(X_s - g(s)\right) \geq 0 \Big| X_t = z \right) \leq \Pr_x \left( \sup_{t\leq s \leq t+h} \left(X_s - g_t(s)\right) \geq 0 \Big| X_t = z \right). \]
    For $z \in (\ell, g(t) - h^{1/4}]$, if the process is to cross the line $g_t(s)$, it must first hit the level $\gm := g(t) - h^{1/4}$ some time during $[t, t+h)$ (we assume without loss of generality that $\km < h^{-3/4}$). Thus for these $z$
        \begin{align*}
            \Pr \left( \sup_{t \leq s \leq t+h} \left(X_s - g_t(s)\right) \geq 0 \Big| X_t = z \right) &\leq \sup_{t \leq t' < t+h} \Pr \left( \sup_{t' \leq s \leq t+h} \left(X_s - g_t(s)\right) \geq 0 \Big| X_{t'} = \gm \right)\\
            &\leq \Pr \left( \sup_{t \leq s \leq t+h} X_s - g_{t,h} \geq 0 \Big| X_{t} = \gm \right)\\
            &= \Pr_{\gm} \left( \sup_{0 \leq s \leq h} X_s \geq g_{t,h} \right),
        \end{align*}
    where $g_{t,h} := g_t(t+h) \equiv g(t) - \km h$. Now let $\tau_h = \tau_h(X) := \inf\{s>0 : X_s \geq g_{t,h}\}$. Then for an arbitrary $\ga \in (\ell, \gm)$ we have
        \begin{align}
        \label{eq:x_notx}
            \Pr_{\gm} \left( \sup_{0 \leq s \leq h} X_s \geq g_{t,h} \right) = \Pr_{\gm} \left( \tau_h \leq h, \inf_{0 \leq s \leq \tau_h} X_s > \ga  \right) + \Pr_{\gm} \left( \tau_h \leq h, \inf_{0 \leq s \leq \tau_h} X_s \leq \ga \right).
        \end{align}

    We bound the last probabilities by considering the associated probabilities for processes which dominate, or are dominated by (as appropriate) $X_s$ almost surely. For the first term on the right-hand side, consider a process $Z_s$ such that $Z_s \geq X_s$ a.s. prior to the time $\tau'$ when $X_s$ leaves the interval $(\ga, g(t))$. Namely, we put
        \[ dZ_s = \mh ds + dW_{s},\mbox{ } s \geq 0, \mbox{ with } \mh := \sup_{\ga \leq y \leq g(t)} \mu(y) < \infty   \]
    due to the continuity of $\mu$ inside the diffusion interval. Then clearly
        \begin{align*}
            \Pr_{\gm} \bigg( \tau_h(X) \leq h, &\inf_{0\leq s\leq \tau_h(X)} X_s > \ga \bigg)\\
                &\leq \Pr_{\gm} \left( \tau_h(Z) \leq h, \inf_{0\leq s\leq \tau_h(X)} Z_s > \ga \right) \leq \Pr_{\gm} \left( \tau_h(Z) \leq h \right)\\
                &= \Pr\left(\sup_{0\leq s\leq h}\left(W_s + \mh s\right) \geq g_{t,h} - \gm \right)\\
                &= \Phi\left( \mh h^{1/2} - (g_{t,h} - \gm) h^{-1/2} \right) + e^{2\mh(g_{t,h} - \gm)} \Phi \left( -\mh h^{1/2} - (g_{t,h} - \gm) h^{-1/2} \right)\\
                &= 1 - \Phi\left(\gb_{-}(\gm)\right) + e^{2\mh(g_{t,h} - \gm)} \bigl[1 - \Phi\left(\gb_{+}(\gm)\right) \bigr],
        \end{align*}
    where the second last equality is a well-known relation for the Brownian motion (see for example (1.1.4) on p.250 in \cite{Borodin_etal_xx02}) and $\gb_{\pm}(z) := (g_{t,h} - z)h^{-1/2} \pm \mh h^{1/2}$. For the second term on the right-hand side in \eqref{eq:x_notx}, we choose a process $Y_s$ such that $Y_s \leq X_s$ a.s. for $s \leq \tau'$. Let
        \begin{align*}
        %\label{eq:lower_proc}
            dY_s = \ms ds + dW_{s}, \mbox{ } s>0, \mbox{ with } \ms := \inf_{\ga \leq y \leq g(t)} \mu(y).
        \end{align*}
    Then clearly
        \begin{align*}
            \Pr_{\gm} \bigg( \tau_h(X) \leq h, &\inf_{0 \leq s\leq \tau_h(X) } X_s \leq \ga \bigg)\\
                &\leq \Pr_{\gm} \left( \tau_h(X) \leq h, \inf_{0 \leq s\leq \tau_h(X) } Y_s \leq \ga \right) \leq \Pr_{\gm} \left( \inf_{0\leq s\leq h} Y_s \leq \ga \right)\\
                &= \Pr \left(\inf_{0\leq s \leq h} \left( \ms s + W_s \right) \leq \gm - \ga \right)\\
                & = 1 - \Phi\left(\gl_{-}(\gm)\right) + e^{2\ms (\gm - \ga)} \bigl[1 - \Phi \left(\gl_{+}(\gm)\right)\bigr],\\
        \end{align*}
    where $\gl_{\pm}(z) := (z-\ga)h^{-1/2} \pm \ms h^{1/2}$. Combining these results gives the bound
        \begin{align}
        \label{eq:combin}
            \Pr \bigg( \sup_{t \leq s \leq t+h} (X_s &- g_t(s)) \geq 0 \Big| X_t = z \bigg)\notag\\
             & \leq  1 - \Phi \left( \gb_{-}(\gm) \right) + e^{2\mh(g_{t,h} - \gm)} \bigl[ 1 - \Phi\left( \gb_{+}(\gm) \right) \bigr]\notag\\
             & \qquad+ 1 - \Phi \left( \gl_{-}(\gm) \right) + e^{2 \ms (\gm - \ga)} \bigl[ 1 - \Phi\left(\gl_{+}(\gm) \right) \bigr]\notag\\
             & \leq \frac{1}{\sqrt{2\pi}} \frac{1}{\gb_{-}(\gm)} e^{-\frac{1}{2} \gb_{-}^2(\gm)} + e^{2\mh (g_{t,h} - \gm)} \frac{1}{\sqrt{2\pi}} \frac{1}{\gb_{+}(\gm)} e^{-\frac{1}{2} \gb_{+}^2(\gm)}\notag\\
             & \qquad + \frac{1}{\sqrt{2\pi}} \frac{1}{\gl_{-}(\gm)} e^{-\frac{1}{2} \gl_{-}^2(\gm)} + e^{2 \ms (\gm - \ga)} \frac{1}{\sqrt{2\pi}} \frac{1}{\gl_{+}(\gm)} e^{-\frac{1}{2} \gl_{+}^2(\gm)},
        \end{align}
    using Mill's inequality.

    Now choose $\ga:= g(t) - h^{1/8}$. Then for the first term on the right-hand side of \eqref{eq:combin} we have
        \begin{align*}
            \gb_{-}(\gm) = (g(t) - \km h - (g(t) - h^{1/4})) h^{-1/2} - \mh h^{1/2} = h^{-1/4} \bigl[1 - (\km + \mh) h^{3/4}\bigr],
        \end{align*}
    and so clearly
        \[  \frac{1}{\gb_{-}(\gm)} e^{-\frac{1}{2} \gb_{-}^2 (\gm)} = o(h)  \]
    as $h \rightarrow 0$. Similarly, as $\gl_{\pm}(\gm) = h^{-3/8} (1 + o(1))$, other terms are also $o(h)$. Thus for $I_1$ in~\eqref{eq:split_tau} we have
        \begin{align}
        \label{eq:final_aft_low}
            I_1 = o(h),
        \end{align}
    and hence the integral $I_1$ does not contribute to the density of $\tau$.

    Next we consider $z \in [g(t) - h^{1/4}, g(t)]$, using essentially the same argument but without conditioning on starting at the point $\gm$, which also eliminates the need to reduce the problem to a level crossing. Returning to the original boundary $g(t)$, this gives
        \begin{align}
        \label{eq:final_aft_up}
            \Pr \bigg( \sup_{t\leq s \leq t+h} (X_s &- g(s)) \geq 0 \Big| X_t = z \bigg) \leq  1 - \Phi \left( \chi_{-}(z) \right) + e^{2(\mh + \km)(g(t) - z)} \bigl[ 1 - \Phi\left( \chi_{+}(z) \right) \bigr]\notag\\
             & \qquad + 1 - \Phi \left( \gl_{-}(\gm) \right) + e^{2 \ms (\gm - \ga)} \bigl[ 1 - \Phi\left(\gl_{+}(\gm) \right) \bigr]\notag\\
             & =  1 - \Phi \left( \chi_{-}(z) \right) + e^{2(\mh + \km)(g(t) - z)} \bigl[ 1 - \Phi\left( \chi_{+}(z) \right) \bigr] + o(h),
        \end{align}
    where $\chi_{\pm}(z) := (g(t) - z)h^{-1/2} \pm (\mh + \km) h^{1/2}$. Returning to \eqref{eq:split_tau} and using \eqref{eq:break_tau_cross}, \eqref{eq:final_before} and \eqref{eq:final_aft_up}, we have
        \begin{align*}
            I_2 &\leq  \int_{g(t)-h^{1/4}}^{g(t)} 2 A_h(g,t) (g(t) - z) \left[1 - \Phi \left( \chi_{-}(z) \right) + e^{2(\mh + \km)(g(t) - z)} \left( 1 - \Phi\left( \chi_{+}(z) \right) \right)\right] dz\\
                 &\qquad+ o(h)\\
                 &=: 2 A_h(g,t) \int_{g(t)-h^{1/4}}^{g(t)} \left( J_1(z) + J_2(z) \right) dz + o(h).
        \end{align*}
    Consider the first term of the integrand (the second is treated in the same manner, since the exponential factor is $(1+ o(1))$). Denoting $(\km + \mh)$ by $\ov{K}$ and changing variables $y = (g(t) - z)/\sqrt{h}$, we have
        \begin{align*}
            \int_{g(t)-h^{1/4}}^{g(t)} J_1(z) dz &= h \int_0^{h^{-1/4}} y \left(1- \Phi(y - \ov{K}\sqrt{h})\right) dy\\
              & \leq h \int_0^{\infty} y \left(1- \Phi(y - \ov{K}\sqrt{h})\right) dy\\
              & = h\left(\int_0^{\infty} (u + \ov{K} \sqrt{h}) (1 - \Phi(u)) du + \int_{-\ov{K}\sqrt{h}}^{0}(u + \ov{K} \sqrt{h} ) (1- \Phi(u)) du\right)\\
              & \leq h \left( \frac{1}{2} \Exp \left(Z^2 ; Z>0 \right) + \ov{K} \sqrt{h} \Exp\left(Z; Z>0 \right) + \ov{K}^2 h \right) = \frac{1}{4}h + o(h),
        \end{align*}
    where $Z$ is a standard normal random variable. One can similarly show that the contribution from $J_2(z)$ will be of the same form. This gives
        \[  I_2 \leq    A_h(g,t) h + o(h).  \]
    Combining the bound with \eqref{eq:split_tau} and \eqref{eq:final_aft_low}, we conclude that
        \[ h^{-1}\Pr_x (\tau \in (t, t+h)) \leq A_h(g,t) + o(1).    \]
    Passing to the limit as $h \rightarrow 0$ in the above expression yields a bound for the density of $\tau$, with the limiting value of $A_h(g,t)$ given by $B(g,t)$ defined in \eqref{eq:Bt_def}. This completes the proof of the theorem.
    \end{proof}

    For diffusions on $\mathbb{R}$ we can extend this result, considering the crossing of a lower boundary. Define
        \[  \un{g}(t) := \min_{0 \leq s \leq t} g(s).   \]
    Then for the random variable $\rho := \inf\{t>0 : X_t \leq g(t)\}$ we have the following assertion.

    \begin{cor}
    \label{cor:lower_den_bound}
        In case {\em [A]}, let $X_t$ and $g(t)$ be as in Theorem~{\em \ref{th:density_bound}}, except with $g(0) < x$. Set
            \[  L_t := \inf_{\un{g}(t) \leq y < \infty } \left(\mu'(y) + \mu^2(y)\right).   \]
        Then $\rho$ has a density $p_{\rho}(t)$ which satisfies
            \[  p_{\rho}(t) \leq C(g,t),    \]
        where
            \[  C(g,t) := \frac{1}{t} (x - g(t) + \kp t) q(t,x,g(t)) e^{G(g(t)) - G(x) -(t/2) L(t)}.    \]
    \end{cor}

    This assertion is a direct consequence of Theorem~\ref{th:density_bound}, applied to the diffusion $-X_t$ crossing the upper boundary $-g(t)$.

    Similar ideas can be used to obtain lower bounds for the density of the first passage time, both for upper and lower boundaries, although again we are restricted to the case [A]. We use the same notation as previously, however instead of $M(t)$ we define $\LM$ by
    \[  \LM := \sup_{-\infty < y \leq \ov{g}(t)} \left(\mu'(y) + \mu^2(y)\right).   \]

    \begin{thm}
    \label{th:density_lower_bound}
        In case {\em [A]}, let $X_t$ and $g(t)$ be as in Theorem~{\em \ref{th:density_bound}}.  Assume that
            \begin{align}
            \label{eq:lower_bound_ineq}
                x < \min_{0 \leq t \leq 1} \left(g(t) - \kp t\right).
            \end{align}
        Then the first passage time $\tau$ of the process $X_t$ of the boundary $g(t)$ has a density $p_{\tau}(t)$ which is bounded from below by
            \begin{align}
            \label{eq:density_lower_bound}
                p_{\tau}(t) \geq \LB,
            \end{align}
        where $\LB$ is given by
            \begin{align*}
                \LB := \frac{1}{t} (g(t) - \kp t - x) q(t,x,g(t)) e^{G(g(t)) - G(x) -(t/2) \LM}.
            \end{align*}
    \end{thm}

    \begin{proof}
        The scheme of proof is the same as for Theorem~\ref{th:density_bound}, with straightforward changes to obtain a lower bound. In particular, instead of considering the linear boundary crossing for $g_t(s)$, we consider similar probabilities for the linear function
            \[  \LG = g(t) + \kp (s-t). \]
        Inequality~\eqref{eq:lower_bound_ineq} arises from the need for this boundary to initially be above $x$.
    \end{proof}

    A similar result holds for lower boundaries as well.

    \begin{cor}
        In case {\em [A]}, let $X_t$ and $g(t)$ be as in Theorem~{\em \ref{th:density_lower_bound}}, except with $g(0) < x$. Define $\LL$ by
            \[  \LL := \inf_{\un{g}(t) \leq y < \infty} \left(\mu'(y) + \mu^2(y)\right).    \]
        If $x > \sup_{0 \leq t \leq 1} \left( g(t) + \km t \right),$ then $\rho$ has a density $p_{\rho}(t)$, which is bounded from below by
            \[  p_{\rho}(t) \geq \LC(g,t),  \]
        where
            \[  \LC(g,t) := \frac{1}{t} (x - g(t) + \kp t) q(t,x,g(t)) e^{G(g(t)) - G(x) -(t/2) \LL(t)}.    \]
    \end{cor}

    Similarly to the case of Corollary~\ref{cor:lower_den_bound}, the above assertion is a direct consequence of Theorem~\ref{th:density_lower_bound}, considering the diffusion $-X_t$ crossing the upper boundary $-g(t)$.

    \begin{rem}
    As we said in Section~\ref{sec:setup}, the choice of the reference processes used is largely arbitrary, within some constraints. If we have further information for a certain diffusion, the above results could be improved by substituting this diffusion as the reference process. In particular, several processes have the same bridges as either Brownian motion or Bessel processes, see \cite{Borodin_0505} and references therein. In other words, if $X_s$ is a process with the same bridge as a Bessel process, we have
        \[  \Pr_x \left(\sup_{0 \leq s \leq t} \left(X_s - g_t(s)\right) < 0 \Big| X_t = z \right) = \Pr_x \left(\sup_{0 \leq s \leq t} \left(R_s - g_t(s)\right) < 0 \Big| R_t = z \right).    \]
    This could eliminate the need to use Lemma~\ref{lem:bridge_prob}, proceeding straight to \eqref{eq:BBridge} or its equivalent Lemma~\ref{lem:BesBridge}. This eliminates the exponential term in \eqref{eq:Bt_def}, however it also relies on knowing the transition density for the diffusion in question.
    \end{rem}

    \begin{rem}
    Further, it is bounds for these conditional probabilities that limit the results obtained above and their accuracy. In particular, Lemma~\ref{lem:BesBridge} gives an excessive bound, of which tightening could improve on the bounds presented here. In case [B] it is also this lemma that needs to be extended to obtain a lower bound on the crossing density and to obtain results for lower boundaries.
    \end{rem}

    Observe that the above results can be extended to some time-inhomogeneous diffusions as well. In general, the transformation leading to a constant diffusion coefficient can result in a drift term which inhibits the use of the above methods. Even assuming we begin with a process with constant diffusion, a similar problem is encountered with the Girsanov density in Lemma~\ref{lem:bridge_prob}.

    However, for some time-inhomogeneous diffusions there is an obvious approach to avoid these complications. If our time-inhomogeneous process $Y_t$ is a function of a time-homogeneous diffusion $X_t$ and time $t$, the idea is to manipulate the boundary crossing probability so as to incorporate the time-inhomogeneity into the boundary. As a simple example, if $Y_t = r(t) X_t$, with $r(t) > 0$, then we can write
        \[  \Pr_x\left(\sup_{0 \leq t \leq 1} (Y_t - g(t)) > 0 \right) = \Pr_x\left(\sup_{0 \leq t \leq 1} \left(X_t - \frac{g(t)}{r(t)}\right) > 0 \right) \]
    and proceed as before.

\section{Approximation Rates for Boundary Crossing\\ Probabilities}
\label{sec:bcp}

    In this section we use the bounds of the previous section to extend the results of \cite{Borovkov_etal_0305} from the Brownian motion case to general diffusions. The problem considered is calculating the probability that a given diffusion will stay between two boundaries over the time interval $[0,1]$. That is, for a diffusion $X_t$ and two functions $g_{\pm}$ on $[0,1]$ such that $\gi(0) < x < \gp(0)$, we need to evaluate
        \begin{align}
        \label{eq:prb_def}
            \Pr_x \left( \gi(t) < X_t < \gp(t), t \in [0,1] \right) =: \Prb{\gi}{\gp}.
        \end{align}
    Again, this can be extended to the time interval $[0,T]$, by a change of scale. We use the same notation for one-sided boundaries by replacing the redundant boundary with $\pm\infty$.

    An exact calculation of the above probability is possible only for a small number of diffusions and boundaries. Therefore in the general case, it is important to have an upper bound for the difference between probability \eqref{eq:prb_def} and the corresponding probability for a pair of simpler boundaries which approximate the original ones and, at the same time, make computing the boundary crossing probability feasible.

    We denote the uniform norm of a function on $[0,1]$ by $||\cdot||$, that is $||g|| = \sup_{0 \leq t \leq 1} |g(t)|$.

    Define $\Dep(g)$ to be the probability that the process $X_t$ crosses the upper boundary $g$ at some point $t \in [0, 1]$ without crossing the boundary $g + \ep$, $\ep>0$, on that time interval:
        \[  \Dep(g) := \Prb{-\infty}{g + \ep} - \Prb{-\infty}{g}.   \]
    Define $\Depm(g)$ similarly for lower boundaries:
        \[  \Depm(g) := \Prb{g - \ep}{\infty} - \Prb{g}{\infty}.    \]

    \begin{thm}
    \label{th:bdy_cross}
        Let $g_{\pm}$ each satisfy condition \eqref{eq:lip_assump}, and consider functions $f_{\pm}$ on $[0,1]$ such that $||g_{\pm} - f_{\pm}|| \leq \ep$ for some $\ep > 0$. For $\tp \in [0,1)$, let
            \[  \over{B}(g) := \sup_{\tp \leq t \leq 1} B(g,t) \;\; \mbox{ and }\;\; \over{C}(g) := \sup_{\tp \leq t \leq 1} C(g,t),    \]
            \begin{align}
            \label{eq:mu_star}
                \mlow := \inf_{\ell < x \leq g(0) + \kp + \ep} \mu(x) \;\; \mbox{ and }\;\; \mup := \sup_{g(0) - \km - \ep \leq x < \infty} \mu(x),
            \end{align}
        and set $\ks := \max\{0, \kp - \mlow\}$ and $\ka := \max\{0, \km + \mup\}$. If $\mlow > -\infty$, then
            \[  \left| \Prbm{-\infty}{\gp} - \Prbm{-\infty}{\fp} \right| \leq \max\left\{ \Dep(\gp), \Dep(\gp - \ep) \right\}, \]
        where
            \begin{align}
            \label{eq:Dep_bound}
                \Dep(g) \leq    \left(\sqrt{\frac{2}{\pi}} \left( \frac{1}{\sqrt{1-\tp}} + 2\over{B}(g) \sqrt{1-\tp} \right) + 2 \ks\right) \ep.
            \end{align}
        In case {\em [A]}, if we also have $\mup < \infty$, then
            \[  \left|\Prbm{\gi}{\gp} - \Prbm{\fm}{\fp}\right| \leq \max\left\{ \Dep(\gp) + \Depm(\gi), \Dep(\gp - \ep) + \Depm(\gi + \ep) \right\},    \]
        where
            \begin{align}
            \label{eq:Depm_bound}
                \Depm(g) \leq \left(\sqrt{\frac{2}{\pi}} \left( \frac{1}{\sqrt{1-\tp}} + 2\over{C}(g)  \sqrt{1-\tp} \right) + 2 \ka\right)\ep.
            \end{align}
    \end{thm}
    \begin{rem}
            (i) Note that the two boundaries $g_{\pm}$ may satisfy condition \eqref{eq:lip_assump} for different constants $K^{\pm}$. For the purposes of this theorem, we have assumed the constants are the same, but if further accuracy is required, this assumption can easily be dropped.

            (ii) The choice of $\tp$ is arbitrary and can be used to minimise the upper bounds for specific diffusion processes and boundaries.

            (iii) By setting
                    \[  \overline{B}(g) := \sup_{0 \leq t \leq 1} B(g,t)    \]
                we obtain the simpler, but less accurate, bound
                    \[  \Dep(g) \leq \left( 4 \sqrt{\frac{\overline{B}(g)}{\pi}} + 2 \ks \right) \ep,   \]
                and similarly for $\Depm(g)$.
    \end{rem}
    \begin{proof}
    We begin by using an argument similar to the one employed in the proof of Theorem~1 in \cite{Borovkov_etal_0305}. Since $||g_{\pm} - f_{\pm}|| \leq \ep$, we have
        \begin{align}
        \label{eq:bound_ff}
            \Prb{\gi+\ep}{\gp-\ep} \leq \Prb{\fm}{\fp} \leq \Prb{\gi - \ep}{\gp + \ep}.
        \end{align}
    Further,
        \[  \Prb{\gi}{\gp} \leq \Prb{\gi - \ep}{\gp + \ep}, \]
    and hence
        \[  0 \leq \left[ \Prb{\gi - \ep}{\gp+ \ep} - \Prb{\gi}{\gp + \ep} \right] + \left[ \Prb{\gi}{\gp+ \ep} - \Prb{\gi}{\gp} \right].   \]
    Each of these terms can now be bounded. For the second term,
        \begin{align*}
            \Prb{\gi}{\gp+ \ep} - \Prb{\gi}{\gp} &= \Pr_x\left( 0 \leq \sup_{0 \leq t \leq 1} \left(X_t - \gp(t) \right) < \ep, \inf_{0 \leq t \leq 1} \left(X_t - \gi(t) \right) >0 \right)\\
                & \leq \Pr_x\left( 0 \leq \sup_{0 \leq t \leq 1} \left(X_t - \gp(t) \right) < \ep \right)\\
                & = \Prb{-\infty}{\gp+ \ep} - \Prb{-\infty}{\gp} = \Dep(\gp),
        \end{align*}
    while for the first term
        \begin{align*}
            \Prb{\gi - \ep}{\gp+ \ep} &- \Prb{\gi}{\gp + \ep}\\
                &= \Pr_x\left( \sup_{0 \leq t \leq 1} \left(X_t - \gp(t) - \ep \right) < 0, -\ep < \inf_{0 \leq t \leq 1} \left(X_t - \gi(t) \right) \leq 0 \right)\\
                & \leq \Pr_x\left( -\ep < \inf_{0 \leq t \leq 1} \left(X_t - \gi(t) \right) \leq 0 \right)\\
                & = \Prb{\gi - \ep}{\infty} - \Prb{\gi}{\infty} = \Depm(\gi).
        \end{align*}
    Hence we have
        \begin{align}
        \label{eq:prb_bound1}
            0 \leq \Prb{\gi - \ep}{\gp + \ep} - \Prb{\gi}{\gp} \leq \Dep(\gp) + \Depm(\gi).
        \end{align}
    Using the same argument gives
        \begin{align}
        \label{eq:prb_bound2}
            0 \leq \Prb{\gi}{\gp} - \Prb{\gi + \ep}{\gp - \ep} \leq \Dep(\gp - \ep) + \Depm(\gi + \ep).
        \end{align}
    Now from \eqref{eq:bound_ff}--\eqref{eq:prb_bound2} we have
        \[  \left|\Prb{\gi}{\gp} - \Prb{\fm}{\fp} \right| \leq \max\left\{ \Dep(\gp) + \Depm(\gi), \Dep(\gp - \ep) + \Depm(\gi + \ep) \right\}. \]
    A simpler argument also gives
        \[  \left|\Prb{-\infty}{\gp} - \Prb{-\infty}{\fp} \right| \leq \max\left\{ \Dep(\gp), \Dep(\gp - \ep) \right\}. \]
    To complete the proof of the theorem, it remains to establish \eqref{eq:Dep_bound} and \eqref{eq:Depm_bound}. Since these bounds are derived in the same way, we will only prove the former one.

    Note that
        \begin{align}
        \label{eq:orig}
            \Dep(g) &= \Pr_x \left( 0 \leq \sup_{0 \leq t\leq 1} \left(X_t - g(t)\right) < \ep \right)\notag\\
                &= \int_0^1 \Pr_x(\tau \in dt) \Pr_x \left(\sup_{t \leq s \leq 1} \left(X_s - g(s)\right) < \ep | X_t = g(t) \right)\notag\\
                &\leq \int_0^1 \Pr_x(\tau \in dt) \Pr \left( \sup_{t \leq s \leq 1} \left(X_s - \kp(s-t)\right) < \ep + g(t) | X_t = g(t) \right)\notag\\
                &= \int_0^1 \Pr_x(\tau \in dt) \Pr_{g(t)} \left( \sup_{0 \leq s \leq 1-t} \left(X_s - \kp s \right) < \ep + g(t) \right).
        \end{align}
    Consider the integrand in the last line. If we have a process $Y_s$ such that $Y_s \leq X_s$ a.s. prior to the time $\inf\{s>0 : X_s \geq g(t) + \ep + \kp s\}$, then
        \[ \Pr_{g(t)} \left( \sup_{0 \leq s \leq 1-t} \left(X_s - \kp s\right) < \ep + g(t) \right) \leq \Pr_{g(t)} \left( \sup_{0 \leq s \leq 1-t} \left(Y_s - \kp s\right) < \ep + g(t) \right).  \]
    To construct a process $Y_s$ meeting the above requirement, simply put
        \[  dY_s = \mlow ds + dW_{s}, \mbox{ } s>0, \mbox{ } Y_0 = g(t),    \]
    with $\mlow$ defined by \eqref{eq:mu_star}. For the new process, we clearly have
        \begin{align*}
            \Pr_{g(t)} \left( \sup_{0 \leq s \leq 1-t} \left(Y_s - \kp s\right) < \ep + g(t) \right) &= \Pr \left( \sup_{0 \leq s \leq 1-t} \left( W_s - (\kp -\mlow)s \right) \leq \ep \right)\\
              &\leq \Pr \left( \sup_{0 \leq s \leq 1-t} \left( W_s - \ks s \right) \leq \ep \right).
        \end{align*}
    From here we continue as per the derivation of relation~(14) in \cite{Borovkov_etal_0305}. As is well-known (see e.g. (1.1.4) on p.250 in \cite{Borodin_etal_xx02}),
        \begin{align*}
            \Pr \bigg( \sup_{0 \leq s \leq 1-t} &\left( W_s - \ks s \right) \leq \ep \bigg)\\
                & = \Phi \left( \ks \sqrt{1-t} + \frac{\ep}{\sqrt{1-t}} \right) - e^{-2\ks \ep} \Phi \left( \ks \sqrt{1-t} - \frac{\ep}{\sqrt{1-t}} \right)\\
                & \leq \Phi \left( \ks \sqrt{1-t} + \frac{\ep}{\sqrt{1-t}} \right) - \Phi \left( \ks \sqrt{1-t} - \frac{\ep}{\sqrt{1-t}} \right) + (1 - e^{-2\ks \ep}).
        \end{align*}
    Since $\Phi'(x) \leq (2 \pi)^{-1/2}$ and $1- e^{-2\ks \ep} \leq 2 \ks \ep$, we conclude that
        \[ \Pr_{g(t)} \left( \sup_{0 \leq s \leq 1-t} \left(X_s - \kp s\right) < \ep + g(t)\right) \leq \sqrt{\frac{2}{\pi (1-t)}}\;\; \ep + 2\ks \ep,  \]
     and therefore \eqref{eq:orig} implies that
        \begin{align}
        \label{eq:mod1}
            \Dep(g) \leq \ep\left(\sqrt{\frac{2}{\pi}} \int_0^1 \frac{\Pr_x(\tau \in dt)}{\sqrt{1-t}} + 2 \ks \right).
        \end{align}
    Fix $\tp \in[0,1)$ and break the integral in \eqref{eq:mod1} as follows:
        \begin{align*}
            \frac{\Dep(g)}{\ep} &\leq \sqrt{\frac{2}{\pi}} \left( \int_0^{\tp} \frac{\Pr_x(\tau \in dt)}{\sqrt{1-t}} + \int_{\tp}^1 \frac{\Pr_x(\tau \in dt)}{\sqrt{1-t}} \right) + 2 \ks\\
              &\leq \sqrt{\frac{2}{\pi}} \left( \frac{1}{\sqrt{1-\tp}} \int_0^{\tp} \Pr_x(\tau \in dt) + \int_{\tp}^1 \frac{\Pr_x(\tau \in dt)}{\sqrt{1-t}} \right) + 2 \ks\\
              &\leq \sqrt{\frac{2}{\pi}} \left( \frac{1}{\sqrt{1-\tp}} + \int_{\tp}^1 \frac{\Pr_x(\tau \in dt)}{\sqrt{1-t}} \right) + 2 \ks.\\
        \end{align*}
    Using the bound \eqref{eq:density_bound} for the density $\Pr_x(\tau \in dt)/dt$, with $\over{B}(g)$ replacing $B(g,t)$, then gives
        \begin{align*}
            \frac{\Dep(g)}{\ep} &\leq \sqrt{\frac{2}{\pi}} \left( \frac{1}{\sqrt{1-\tp}} + \over{B}(g)  \int_{\tp}^1 \frac{dt}{\sqrt{1-t}} \right) + 2 \ks\\
            & = \sqrt{\frac{2}{\pi}} \left( \frac{1}{\sqrt{1-\tp}} + 2\over{B}(g)  \sqrt{1-\tp} \right) + 2 \ks.
        \end{align*}
    The derivation of \eqref{eq:Depm_bound} uses the same argument. Theorem~\ref{th:bdy_cross} is proved.
    \end{proof}

    \begin{cor}
        For a Borel set $B$, define
            \[  \Prbbm{\gi}{\gp} := \Pr_x\left(\gi(t) < X_t < \gp(t), t\in[0, 1]; X_1 \in B \right).    \]
        Then, under the assumptions of Theorem~{\em \ref{th:bdy_cross}}, we have in the general case
            \[  | \Prbbm{-\infty}{\gp} - \Prbbm{-\infty}{\fp}| \leq \max\left\{ \Dep(\gp), \Dep(\gp - \ep) \right\}, \]
        while in case {\em [A]}
            \[  \left|\Prbbm{\gi}{\gp} - \Prbbm{\fm}{\fp}\right| \leq \max\left\{ \Dep(\gp) + \Depm(\gi), \Dep(\gp - \ep) + \Depm(\gi + \ep) \right\},  \]
        where $\Dep$ and $\Depm$ satisfy the bounds given in Theorem~{\em \ref{th:bdy_cross}}.
    \end{cor}

    \begin{proof}
        Essentially the same argument as for Theorem~\ref{th:bdy_cross} gives the desired result.
    \end{proof}

    As a simple consequence, we have an analogue to Corollary~1 of \cite{Borovkov_etal_0305}, which shows that using approximation by piecewise-linear functions results in a  rate of convergence that is a quadratic function of the partition rank.

    \begin{cor}
        Let $g_{\pm}$ be continuously differentiable on $[0,1]$, and $g'_{\pm}$ be absolutely continuous, satisfying $|g''_{\pm}| \leq \xi < \infty$ almost everywhere. If $0 = t_0 < t_1 \cdots < t_n = 1$ is a partition of $[0,1]$ of rank $\gd = \max_{0 < i \leq n} |t_i - t_{i-1}|$, and $f_{\pm}$ are piecewise linear with nodes at the points $(t_i, g_{\pm}(t_i))$, then
            \[  \left| \Prbm{-\infty}{\gp} - \Prbm{-\infty}{\fp} \right| \leq \frac{C}{8} \xi \gd^2.    \]
        In case {\em [A]} we also have
            \[  \left|\Prbm{\gi}{\gp} - \Prbm{\fm}{\fp}\right| \leq \frac{D}{8} \xi \gd^2.  \]
        where $C$ and $D$ can be bounded explicitly using Theorem~{\em \ref{th:bdy_cross}}. In particular, for a uniform partition with $t_i = i/n$, $0 \leq i \leq n$, the convergence rate of the above probabilities is ${\cal O}(n^{-2})$.
    \end{cor}

    \begin{proof}
        Using Theorem~\ref{th:bdy_cross}, we need only to show that
            \[  \ep = ||f - g|| \leq \frac{1}{8} \xi \gd^2, \]
        the proof of which can be found on p.90 of \cite{Borovkov_etal_0305}.
    \end{proof}

    Note that a similar result holds for the probabilities $\Prbb{\fm}{\fm}$.

\section{Examples}
\label{sec:examples}

    In this section we illustrate the precision of the results from the previous two sections. Our bounds from Theorems~\ref{th:density_bound} and~\ref{th:density_lower_bound} for the first passage time density are compared to some known densities for the Brownian motion, Bessel processes and Ornstein-Uhlenbeck processes. We also illustrate by numerical examples the results of Section~\ref{sec:bcp} and, and in the case of the Brownian motion, compare these with known results from the literature.

    \subsection{Brownian Motion}

    First we consider a Brownian motion starting at zero. Observe that for a linear boundary the bound of Theorem~\ref{th:density_bound} is an exact result. Next we consider a curvilinear Daniels boundary (as it has a known first passage time density) given by
        \begin{align}
        \label{eq:dan_bdy}
            g(t) = \gd - \frac{t}{2\gd} \log \left( \frac{\gk_1}{2} + \sqrt{\frac{1}{4} \gk_1^2 + \gk_2 e^{-4\gd^2/t} } \right),
        \end{align}
    where $\gd \ne 0$, $\gk_1 >0$ and $\gk_2 \in \mathbb{R}$ subject to $\gk_1^2 + 4 \gk_2 >0$. For this boundary, the first passage time density is given by
        \begin{align}
        \label{eq:dan_fpt}
            p_{\tau}(t) = \frac{1}{\sqrt{2 \pi} t^{3/2}} \left( \gd \gk_1 e^{-(g(t) - 2\gd)^2/(2t)} + 2\gd\gk_2 e^{-(g(t) - 4\gd)^2/(2t)}\right)
        \end{align}
    (see \cite{Daniels_0696} for further information). We choose for our example the parameter values $\gd = \gk_1 = \gk_2 = 0.5$. Figure~\ref{fig:dan_bdy} displays graphs of the boundary~\eqref{eq:dan_bdy} and the first passage time density~\eqref{eq:dan_fpt} compared to the upper and lower bounds as calculated by~\eqref{eq:density_bound} and \eqref{eq:density_lower_bound} respectively.

\begin{figure}[!htb]
    \begin{centering}$
        \begin{array}{cc}
            \includegraphics[width = 0.45 \textwidth, height = 2.5 in]{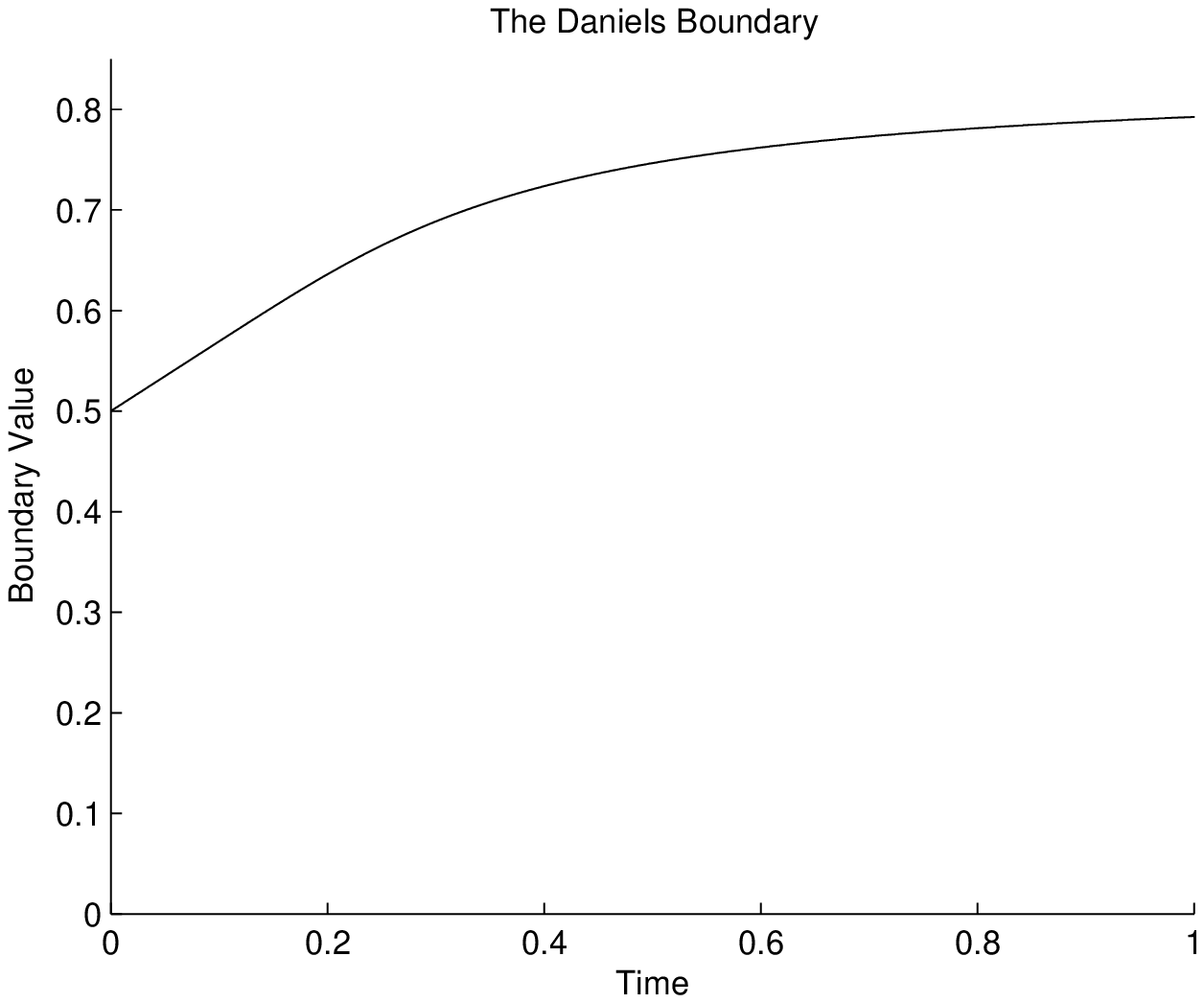} &
            \includegraphics[width = 0.45 \textwidth, height = 2.5 in]{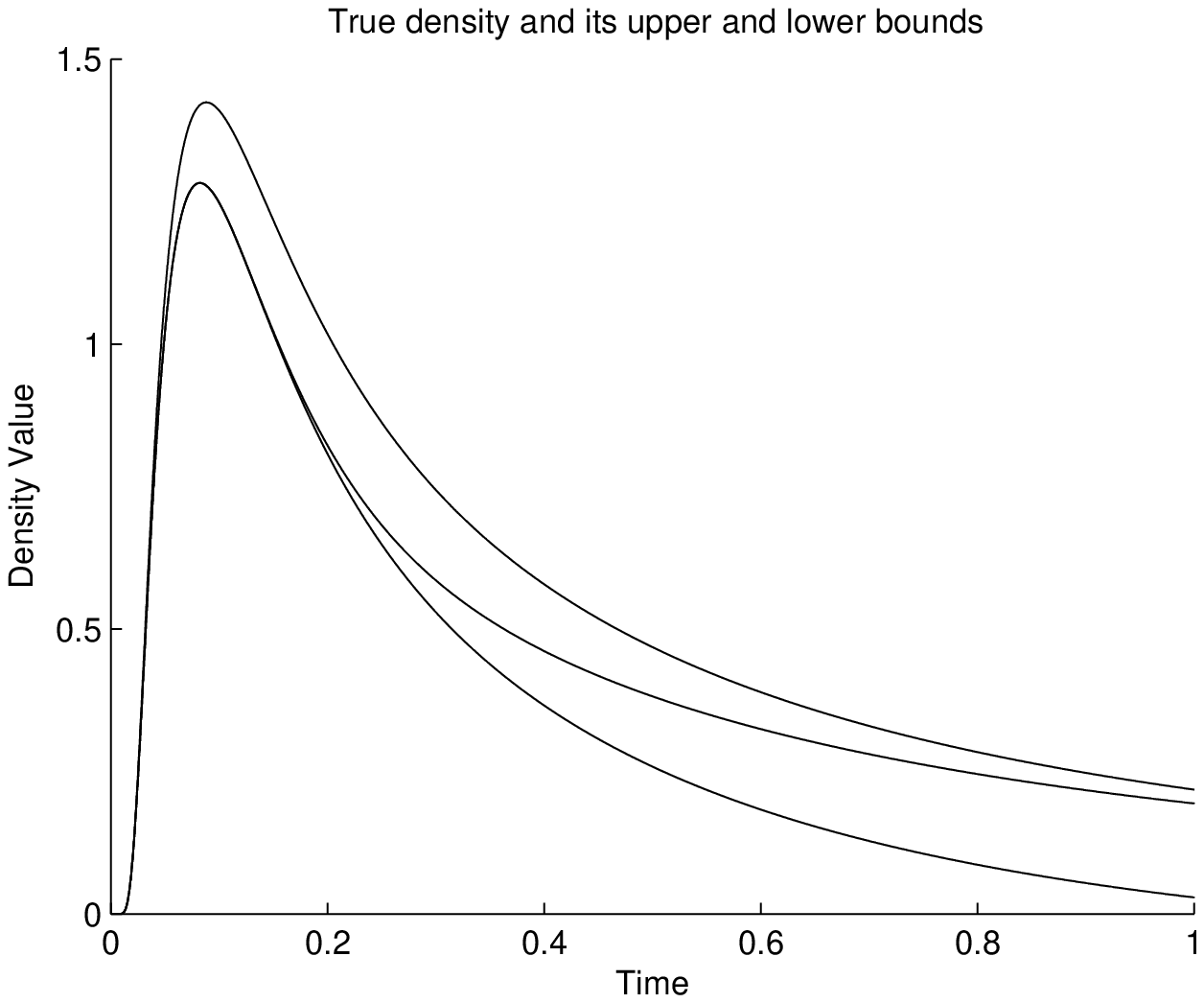}
        \end{array}$
        \caption{The Daniels boundary with $\gd =\gk_1 = \gk_2 = 0.5$, alongside its true first passage time density and upper and lower bounds.}
    \end{centering}
    \label{fig:dan_bdy}
\end{figure}

    We can approximate the Daniels boundary by a linear function on $[0,1]$ with a uniform error of approximately $0.05$, and clearly a piecewise linear boundary will reduce this. The results from Section~\ref{sec:bcp} for an approximating function $f$ with $||f-g|| \leq \ep$ give the bound
        \[  \left|\Prb{-\infty}{g} - \Prb{-\infty}{f}\right| \leq 3.05 \ep \]
    with $\tp \approx 0.555$. By way of comparison, the results in \cite{Borovkov_etal_0305} bound the above difference by $3.73 \ep$.

    Note that choosing different parameter values to make the boundary flatter can increase the accuracy of the result. For example, with the choice $\gd = 1.0, \gk_1 = \gk_2 = 0.9$ we have a maximum relative error for the density upper bound of less than $2.9\%$ and the difference of the boundary crossing probabilities is bounded by $1.7\ep$.

    \subsection{The Bessel Process of Dimension 4}

    For a Bessel process of dimension 4, we choose the straight line boundary $g(t) = 1$, and a starting point $x = 0.5$. The graph of the true first passage time density (given by (2.0.2) in \cite{Borodin_etal_xx02}), as well as the upper bound calculated using Section~\ref{sec:fpt_density}, is shown on a semi-log scale in Figure~\ref{fig:bes_bound}.

\begin{figure}[!htb]
    \begin{centering}
            \includegraphics[width = 0.45 \textwidth, height = 2.5 in]{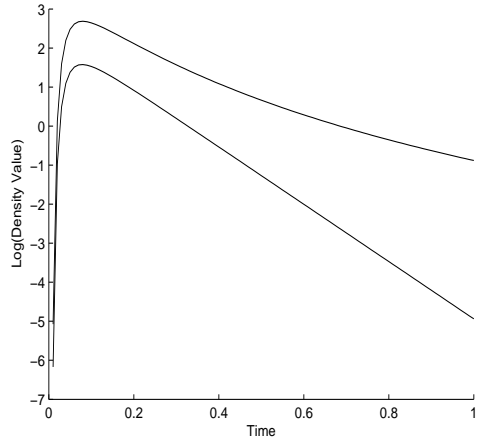}
        \caption{First passage time density for a Bessel process and a straight line boundary, alongside its upper bound on a semi-log scale.}
        \label{fig:bes_bound}
    \end{centering}
\end{figure}

    As mentioned, improvements in Lemma~\ref{lem:BesBridge} will help to improve this result, which anyway gives a useful bound on the error in the boundary crossing probabilities. Using a value of $\tp \approx 0.75$ gives the bound, for an approximating function $f$, with $||f-g|| \leq \ep \leq 0.05$,
        \[  \left|\Prb{-\infty}{g} - \Prb{-\infty}{f}\right| \leq 2.25 \ep. \]

    \subsection{The Ornstein-Uhlenbeck Process}

    Now consider an Ornstein-Uhlenbeck process $S_t$, which satisfies the SDE
        \[  dS_t = -\gt (S_t - \mu) dt + dW_t,  \]
    where $\gt > 0$, $\mu \in \mathbb{R}$. According to \cite{Buonocore_etal_1287}, for the `hyperbolic' boundary
        \[  g(t) = \mu + A e^{-\gt t} + B e^{\gt t},    \]
    where $A, B \in \mathbb{R}$ are such that $x < \mu + A + B$, the first passage time density is given by
        \begin{align}
        \label{eq:OU_true}
            p_{\tau}(t) = 2 \gt \frac{ \left| A + B - x + \mu \right|}{e^{\gt t} - e^{-\gt t}}p(t, x, g(t)),
        \end{align}
    where $p(t, x, z)$ is the transition density of the process $S_t$:
        \[  p(t,x,z) = \left( \frac{\gt e^{2\gt t}}{\pi (e^{2\gt t}-1)}\right)^{1/2} \exp\left(\frac{\gt\left((z -\mu)e^{\gt t} - (x-\mu)\right)^2}{1-e^{2\gt t}} \right).  \]
    For the parameter choice $\gt = 0.5$, $\mu = 2$, $x = 1$, $A = 1$ and $B = -1$, Figure~\ref{fig:OU_bound} displays the true first passage time density \eqref{eq:OU_true} compared to the upper bound calculated by \eqref{eq:density_bound}.

\begin{figure}[!htb]
    \begin{centering}
            \includegraphics[width = 0.45 \textwidth, height = 2.5 in]{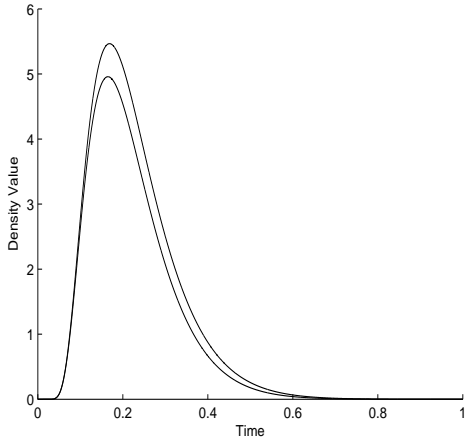}
        \caption{First passage time density of a hyperbolic boundary for an Ornstein-Uhlenbeck process, alongside its upper bound.}
        \label{fig:OU_bound}
    \end{centering}
\end{figure}

    Although this process is an example of case [B], we are unable to obtain a meaningful lower bound as $\LM$ is unbounded.

    Using the results of Section~\ref{sec:bcp} with $\tp \approx 0.28$ gives the bound
        \[  \left|\Prb{-\infty}{g} - \Prb{-\infty}{f}\right| \leq 1.88 \ep, \]
    when $||f-g|| \leq \ep \leq 0.05$.

    {\bf Acknowledgement:} This research was supported by the ARC Centre of Excellence for Mathematics and Statistics of Complex Systems.

\end{document}